\documentclass[12pt,a4paper,reqno]{amsart}
\usepackage[hmargin=2.5cm,bmargin=2.5cm,tmargin=3cm]{geometry}
\usepackage{enumerate}
\usepackage{amsmath,amsthm,amssymb,amsfonts,latexsym}
\usepackage[gen]{eurosym}
\usepackage[german,english]{babel}
\usepackage{mathtools}
\usepackage{esint}
\usepackage{hyperref}
\usepackage{tikz}
\allowdisplaybreaks
\usepackage{latexsym}

\usepackage{tikz}

\newtheorem{theorem}{Theorem}
\newtheorem{lemma}[theorem]{Lemma}

\newtheorem{proposition}[theorem]{Proposition}

\newtheorem{remark}[theorem]{Remark}

\newcommand{\be}{\begin{equation}}
\newcommand{\ee}{\end{equation}}
\newcommand{\bea}{\begin{eqnarray*}}
	\newcommand{\eea}{\end{eqnarray*}}
\newcommand{\beq}{\begin{eqnarray}}
\newcommand{\eeq}{\end{eqnarray}}

\newcommand{\mV}{\mathsf V}

\newcommand{\mv}{\mathsf v}
\newcommand{\mE}{\mathsf E}
\newcommand{\me}{\mathsf e}

\allowdisplaybreaks

\usepackage{todonotes}

\title[Some spectral comparison results on infinite quantum graphs]{Some spectral comparison results on infinite quantum graphs} 

\subjclass[2010]{}

\keywords{}

\author[P.~Bifulco]{Patrizio Bifulco}
\author[J.~Kerner]{Joachim Kerner}

\address{Lehrgebiet Analysis, Fakult\"at Mathematik und Informatik, Fern\-Universit\"at in Hagen, D-58084 Hagen, Germany}
\email{patrizio.bifulco@fernuni-hagen.de}

\address{Lehrgebiet Analysis, Fakult\"at Mathematik und Informatik, Fern\-Universit\"at in Hagen, D-58084 Hagen, Germany}
\email{joachim.kerner@fernuni-hagen.de}

\date{\today}

\thanks{
}

\begin{document}

	\begin{abstract} In this paper we establish spectral comparison results for Schrödinger operators on a certain class of infinite quantum graphs, using recent results obtained in \cite{BifulcoKerner,BSSS} for finite graphs. We also show that new features do appear on infinite quantum graphs such as a modified local Weyl law. In this sense, we regard this paper as a starting point for a more thorough investigation of spectral comparison results on more general infinite metric graphs. 
	\end{abstract}
	
\maketitle

\section{Introduction}

Metric or quantum graphs have become a popular model in many areas of pure and applied sciences. They are obtained by starting with a discrete graph and associating a length to each edge thereof. In this way one obtains a collection of intervals that are glued together in the vertices of the graph and since one is left with intervals, methods of classical one-dimensional analysis can be applied. Nevertheless, since the geometry of a graph is generally more complex than that of an interval, new interesting features appear. In this sense, metric graphs interpolate between one-dimensional and higher-dimensional aspects as studied, for example, in the setting of manifolds. For a more general introduction to quantum graphs we may refer to \cite{BerkolaikoBook,Berkolaiko,Mugnolo} and references therein. 

In any case, most of the literature on quantum graphs is concerned with so-called finite graphs which means that the underlying discrete graph has only a finite number of vertices and edges; in many cases, one also assumes in addition that the length of each edge is finite since then the metric graph is also compact. Nevertheless, the study of infinite quantum graphs has become increasingly popular and we shall refer to \cite{AKM,ExnerInfinite,KNEstimates,KostenktoNicolussi} and references therein for more information about this exciting research direction. As already suggested by title, also in this paper we are concerned with Schrödinger operators on infinite quantum graphs. More explicitly, our main aim is to translate recent spectral comparison results for finite graphs obtained in \cite{BSSS,BifulcoKerner} to a certain class of infinite quantum graphs. For the sake of readability, we start with a simple but \textit{infinite} quantum graph which has total finite length and derive our main comparison results for such a graph first. Subsequently, this also allows for a more direct comparison with results obtained in~\cite{EggerSteinerInfinite} for a simple infinite quantum graph but with infinite length. As will become clear in the following, our spectral comparison results are strongly intertwined with other important concepts such as the heat kernel and the standard Weyl asymptotics for the eigenvalue counting function which again is linked to the behaviour of the heat kernel at small times via Karamata's Tauberian theorem. Therefore, as soon as those objects show a different behaviour than the classical and expected one, new interesting results will follow as illustrated in Section~\ref{ModifiedWeyl}. In this sense, we regard this paper as a first step towards a more thorough investigation of spectral comparison results on more general infinite quantum graphs such as, for example, infinite trees~\cite{FrankKovarikTrees} or comb-like graphs~\cite{MugnoloTaeuferInfinite}.

The paper is organized as follows: In Section~\ref{TheSetting} we start with a specific infinite quantum graph, the path graph of finite length, and formulate (rather standard) results which are relevant for the proof of the main results. In Section~\ref{MainResults} we then formulate our main comparison results for the infinite path graph of finite length (Theorem~\ref{MainResult0} and Theorem~\ref{MainResult}); the proof of Theorem~\ref{MainResult} is given in Section~\ref{ProofMainResult}. In Section~\ref{ModifiedWeyl} we then derive a modified local Weyl law for a certain infinite path graph of infinite length. Finally, in Section~\ref{MainResultsII} we generalize the main comparison results to a larger class of infinite quantum graphs (Theorem~\ref{GenI} and Theorem~\ref{GenII}).

\section{The setting}\label{TheSetting}
\footnote{The construction of the Hamiltonian as well as the spectral properties described in this section are rather standard but we include them for completeness, see, e.g., \cite{Brasche,SSTrace,AKM} and references therein for related considerations} Starting with the interval $I:=[0,L]$, $L > 0$, we build an infinite metric graph $\mathcal{G}$ over $I$ whose vertex set $\mV_{\mathcal{G}}$ consists of the two outer vertices $\mv_1=0$ and $\tilde{\mv}=L$ and a sequence $(\mv_n)_{n \geq 2} \subset [0,L)$, $\mv_n < \mv_{n+1}$ and such that 
\begin{equation}
    \sum_{n=1}^{\infty} (\mv_{n+1}-\mv_{n})=L\ ,
\end{equation}
see also Figure \ref{fig:path-graph} right below.
\begin{figure}[h]
\begin{tikzpicture}[scale=0.50]
      \tikzset{enclosed/.style={draw, circle, inner sep=0pt, minimum size=.10cm, fill=gray}}

      \node[enclosed, label={below: $\mv_3$}] (C) at (8,4) {};
      \node[enclosed, label={below: $\mv_1$}] (A) at (-1,4) {};
      \node[enclosed, label={below: $\mv_2$}] (B) at (4,4) {};
      \node[enclosed, label={below: $\mv_4$}] (D) at (11,4) {};
      \node[enclosed,] (E) at (13,4) {};
      \node[enclosed, label={below: $L$}, fill = white] (F) at (14,4) {};

      \draw (A) -- (B) node[midway, above] (edge1) {};
      \draw (B) -- (C) node[midway, above] (edge2) {};
      \draw (C) -- (D) node[midway, above] (edge3) {};
      \draw (D) -- (E) node[midway, above] (edge4) {};
      \draw (E) edge[dotted] (F) node[midway, above] (edge4) {};
     \end{tikzpicture}
     \vspace{-1.5cm}
     \caption{The infinite graph $\mathcal{G}$ on infinite vertices $\mv_n$ and finite length $L$.}\label{fig:path-graph}
     \end{figure}
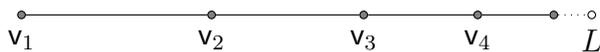
     
On the graph Hilbert space $L^2(\mathcal{G})=L^2(I)$ we then introduce the quadratic form 
\begin{equation}\label{DefForm}
    h^{\sigma}[\varphi]=\int_I |\varphi^{\prime}|^2\ \mathrm{d}x+\sum_{n=1}^{\infty}\sigma_n |\varphi(\mv_n)|^2\ ,
\end{equation}
where $(\sigma_n)_{n \in \mathbb N} \subset [0,\infty)$ and with form domain 
$$\mathcal{D}_{\sigma}:=\{\varphi \in H^1(I):\ h^{\sigma}[\varphi] < \infty \}\ .$$
\begin{theorem}\label{TheoremForm} The form $\left(h^{\sigma}[\cdot],\mathcal{D}_{\sigma}\right)$ is densely defined, closed and positive.
\end{theorem}
\begin{proof} 
Positivity readily follows from the definition. Since $C^{\infty}_0(0,L) \subset  \mathcal{D}_{\sigma}$, the form is densely defined. 

Now, let $(\varphi_n)_{n \in \mathbb{N}} \subset \mathcal{D}_{\sigma}$ be a Cauchy sequence with respect to the form norm $\|\cdot\|^2_{h^\sigma}:=h^{\sigma}[\cdot]+\|\cdot\|^2_{L^2(I)}$. Then, by~\eqref{DefForm}, $(\varphi_n)_{n \in \mathbb{N}}$ is also a Cauchy sequence in $H^1(I)$ and hence converges to a function $\varphi \in H^1(I)$. Employing Fatou's lemma we then conclude $\varphi \in \mathcal{D}_{\sigma}$ and hence the statement.
\end{proof}
\begin{remark} Whenever one has $\sigma \in \ell^1(\mathbb{N})$ one infers that $\mathcal{D}_{\sigma}=H^1(I)$ since functions in $H^1(I)$ have a bounded representative, in other words, $H^1(\mathcal{G}) \hookrightarrow L^\infty(\mathcal{G})$.
\end{remark}

According to Theorem~\ref{TheoremForm} and the representation theorem of forms, there exists a unique self-adjoint operator associated with $\left(q_{\sigma}[\cdot],\mathcal{D}_{\sigma}\right)$. Informally, this operator can be written as
\begin{equation}
H_{\sigma}=-\frac{\mathrm{d}^2}{\mathrm{d}x^2}+\sum_{n=1}^{\infty}\sigma_n \delta(x-\mv_n)\ .
\end{equation}
Note that the $\delta$-potentials in the vertices lead to discontinuities of the derivatives, while the function still remains continuous. More explicitly, at a given vertex $\mv_n$, the sum of all (inward normal, where inward means into the adjacent edge) derivatives equals the value of the function in the vertex times $\sigma_n$. Such matching conditions are often called \emph{$\delta$-coupling conditions}. 

In any case, since the form domain $\mathcal{D}_{\sigma}$ is compactly embedded in $L^2(I)$, we immediately obtain 
\begin{proposition} The operator $H_{\sigma}$ has purely discrete spectrum.
\end{proposition}
We shall denote the eigenvalues of $H_{\sigma}$ by 
\begin{equation*}
0 \leq \lambda_1(\sigma) \leq \lambda_2(\sigma) \leq \dots\ ,
\end{equation*}
counting them with multiplicity and we introduce the eigenvalue counting function
\begin{equation*}
N_{\sigma}(\lambda):=\{n\in \mathbb{N}: \lambda_n(\sigma) \leq \lambda\} \ .
\end{equation*}
\begin{theorem}[Weyl's law]\label{WeylsLaw} One has
\begin{equation}\label{WeylLaw}
\lim_{\lambda \rightarrow \infty}\frac{N_{\sigma}(\lambda)}{\sqrt{\lambda}}=\frac{L}{\pi}
\end{equation}
\end{theorem}
\begin{proof} The statement follows from a suitable operator bracketing argument: More explicitly, one has
\begin{equation*}
N_{\sigma=\infty}(\lambda) \leq N_{\sigma}(\lambda) \leq N_{\sigma=0}(\lambda)\ ,
\end{equation*}
where $\sigma=\infty$ referes to Dirichlet boundary conditions in the vertices of the graph. It is well-known that $N_{\sigma=0}(\lambda)$ satisfies~\eqref{WeylLaw} since in this case we are simply left with an Laplacian on an interval of length $L$ and subject to Neumann boundary conditions. Regarding $N_{\sigma=\infty}(\lambda)$ we observe that, setting $L_n:=\mv_{n+1}-\mv_n$, 
\begin{equation*}
N_{\sigma=\infty}(\lambda) =\#\left\{(m,n): \frac{\pi^2 m^2}{L^2_n} \leq \lambda \right\}\ .
\end{equation*}
We can rewrite $N_{\sigma=\infty}(\lambda)$ as
\begin{equation*}
N_{\sigma=\infty}(\lambda)=\sum_{n: \frac{\pi}{\sqrt{\lambda}}\leq L_n} \bigg\lfloor \frac{\sqrt{\lambda} L_n}{\pi} \bigg\rfloor
\end{equation*}
and therefore 
\begin{equation*}
\lim_{\lambda \rightarrow \infty}\frac{N_{\sigma=\infty}(\lambda)}{\sqrt{\lambda}}=\frac{L}{\pi}\ . \qedhere
\end{equation*}
\end{proof}
\begin{remark} Theorem~\ref{WeylsLaw} establishes the standard Weyl law one knows from the finite setting and this has to do with the fact that the infinite graph of consideration has total finite length. On the other hand, as will become clear in Section~\ref{MainResults}, Weyl's law plays a central role in establishing the main spectral comparison result. At this point, let us also already refer to the paper \cite{EggerSteinerInfinite} where the authors also study an infinite path graph but with infinite total length (see also \cite{AKM} for related models). In this paper they show that, although the spectrum of the Hamiltonian is still purely discrete, one has non-standard Weyl asymptotics for the eigenvalue counting function (strictly speaking, they show this for Dirichlet boundary conditions in the vertices). Therefore, some things are likely to change for general infinite graphs as elaborated on in Section~\ref{ModifiedWeyl} in more detail.
\end{remark}
\section{Main results I: The infinite path graph of finite length}\label{MainResults}
As described above, the main goal of this paper is to compare the spectrum of two different Schrödinger operators on a given (infinite) graph $\mathcal{G}$ in a suitable sense. More explicitly, we are interested in the mean value of eigenvalue distances. Hence, for the infinite path graph of finite length, we aim to study the quantity
\begin{equation}\label{EigenvalueDifferences}
    \frac{1}{N}\sum_{n=1}^{N}\left(\lambda_n(\sigma)-\lambda_n(0)\right)\ 
\end{equation}
in the limit $N \rightarrow \infty$. Note that a related quantity was investigated in \cite{RWY} for domains in $\mathbb{R}^2$ and in \cite{RR,BSSS,BifulcoKerner} for finite quantum graphs. 

On an infinite graph $\mathcal{G}$ and contrary to the finite setting, it is not clear that~\eqref{EigenvalueDifferences} remains finite in the limit of $N \rightarrow \infty$. This is indeed supported by the following statement. 
\begin{theorem}\label{MainResult0} Consider the operator $H_{\sigma}$ with $(\sigma_n)_{n \in \mathbb{N}} \notin \ell^1(\mathbb{N})$. Then
\begin{equation*}
\lim_{N \rightarrow \infty}\frac{1}{N}\sum_{n=1}^{N}\left(\lambda_n(\sigma)-\lambda_n(0)\right)=\infty\ .
\end{equation*}
\end{theorem}

\begin{proof} We introduce the cut-off sequence $\sigma_M:=\left(\hat{\sigma}_n\right)_{n \in \mathbb{N}}$ via
\begin{equation*}
\hat{\sigma}_n:=\begin{cases} \sigma_n \quad \text{for}\quad 1\leq n\leq M\ , \\
0 \quad \text{else}\ .
\end{cases}
\end{equation*}
Then, \cite[Theorem~1]{BifulcoKerner} (resp.\ \cite[Theorem~1.3]{BSSS}) implies that 
\begin{equation*}
\lim_{N \rightarrow \infty}\frac{1}{N}\sum_{n=1}^{N}\left(\lambda_n(\sigma_M)-\lambda_n(0)\right)=\frac{2\sigma_1}{L}+\frac{\sum_{n=2}^{\infty}\hat{\sigma}_n}{L}\ .
\end{equation*}
On the other hand, by a straightforward operator bracketing argument, one also has that the sequence
\begin{equation*}
    \left( \frac{1}{N}\sum_{n=1}^{N}\left(\lambda_n(\sigma_M)-\lambda_n(0)\right) \right)_{M \in \mathbb{N}}
\end{equation*}
is increasing and hence we conclude the statement.
\end{proof}
Of course, now it remains to check the case where $(\sigma_n)_{n \in \mathbb{N}} \in \ell^1(\mathbb{N})$. We obtain the following result whose proof is given in the subsequent section.
\begin{theorem}\label{MainResult} Consider the operator $H_{\sigma}$ with $(\sigma_n)_{n \in \mathbb{N}} \in \ell^1(\mathbb{N})$. Then
\begin{equation*}
\lim_{N \rightarrow \infty}\frac{1}{N}\sum_{n=1}^{N}\left(\lambda_n(\sigma)-\lambda_n(0)\right)=\frac{2\sigma_1}{L}+\frac{\sum_{n=2}^{\infty}\sigma_n}{L}
\end{equation*}
\end{theorem}

\section{Proof of Theorem~\ref{MainResult}}\label{ProofMainResult}

In principle, the proof follows the same path as the corresponding proof in the finite setting. In a first step, we use an explicit expression for the derivative of the eigenvalue curve $\lambda_n(\tau \sigma)$, $\tau \in [0,1]$. More explicitly, we obtain
\begin{equation}\label{EquationOneProof}\begin{split}
    \frac{1}{N}\sum_{n=1}^{N}\left(\lambda_n(\sigma)-\lambda_n(0)\right)&=\frac{1}{N}\sum_{n=1}^{N} \int_{0}^{1}\frac{\mathrm{d}\lambda_n(\tau \sigma)}{\mathrm{d}\tau}  \mathrm{d}\tau \ ,\\
    &=\frac{1}{N}\sum_{n=1}^{N} \int_{0}^{1}\left(\sum_{m=1}^{\infty}\sigma_m |f^{\tau \sigma}_n(\mv_m)|^2\right)\mathrm{d}\tau\ , \\
    &= \int_{0}^{1}\left(\sum_{m=1}^{\infty}\sigma_m \frac{1}{N}\sum_{n=1}^{N}|f^{\tau \sigma}_n(\mv_m)|^2\right)\ \mathrm{d}\tau \ ,
    \end{split}
\end{equation}
where $f^{\tau \sigma}_n \in H^1(I)$ is the normalized $n$-th eigenfunction to the eigenvalue $\lambda_n(\tau \sigma)$.
In order to pull the limit $N \rightarrow \infty$ into the integral, we make use of
\begin{lemma}\label{Lemma1}  There exists a constant $C > 0$ independent of $x \in \mathcal{G}$ and $\tau \in [0,1]$ such that 
\begin{equation*}
    \frac{1}{N}\sum_{n=1}^{N}|f^{\tau \sigma}_n(x)|^2 \leq C\ .
\end{equation*}
\end{lemma}

\begin{proof}[Proof of Lemma~\ref{Lemma1}] Similar as in \cite[(22)]{BifulcoKerner}, for $x \in \mathcal{G}$, we first estimate
\begin{equation*}\begin{split}
    \sum_{n=1}^N |f^{\tau \sigma}_n(x)|^2 & \leq \sum_{n=1}^N \mathrm{e}^{1-\frac{\lambda_n(\sigma)}{\lambda_N(\infty)}}\vert f_n^{\tau \sigma}(x) \vert^2\\ &= \mathrm{e} \sum_{n=1}^N \mathrm{e}^{-\frac{\lambda_n(\sigma)}{\lambda_N(\infty)}} \vert f_n^{\tau \sigma}(x) \vert^2 \nonumber \\&\leq \mathrm{e} \cdot  p_{\tau \sigma}\bigg(\frac{1}{\lambda_N(\infty)};x,x\bigg)\ ,
    \end{split}
\end{equation*}
	where we made use of the following relation for the so-called \emph{heat kernel} associated with the Schrödinger operator $H_{\tau \sigma}$,
	\begin{equation}\label{RelationXXX}
	p_{\tau \sigma}(t;x,x) = \sum_{n=1}^\infty \mathrm{e}^{-t\lambda_n(\tau \sigma)} \vert f_n^{\tau \sigma}(x)\vert^2\ .
	\end{equation}
 We then use the inequality (using results from \cite{Ouh05}, cf. also \ \cite[Appendix~A]{BifulcoKerner})
\begin{equation}\label{BracketingI}
p_{\sigma=(\infty, \tau\sigma_{\mv}, \infty)}(t;x,x) \leq p_{\tau \sigma}(t;x,x) \leq p_{\sigma=0}(t;x,x)\ , \quad x \in \mathcal{G}\ , 
\end{equation}
where $p_{\sigma=(\infty, \tau\sigma_{\mv}, \infty)}(t;x,x)$ denotes the heat kernel corresponding to the Schrödinger operator $H_{\sigma=(\infty, \tau\sigma_{\mv}, \infty)}$ where each $\delta$-coupling condition is replaced by a Dirichlet condition except for the vertex $\mv \in \mV$ with $\mathrm{dist}(x, \mv) = \min_{n \in \mathbb{N}} \mathrm{dist}(x, \mv_n)$. Another interesting aspect about this inequality is that $ p_{\sigma=0}(t;\cdot,\cdot)$ is simply the heat kernel of the Laplacian on a finite interval of length $L > 0$ and subject to Neumann boundary conditions which implies, in particular, that its properties are well-known (see, for instance, \cite{BER}). Similarly, for an $x \in \mathcal{G}$ with $x \neq \mv_m$ for all $m \in \mathbb N$ (that is, $x$ is not a vertex of $\mathcal{G}$), $p_{\sigma=\infty}(t;x,x)$ is nothing else than the heat kernel of a Dirichlet Laplacian on the associated (maybe small) interval and hence its properties are also well-known (one should also take into account that the Dirichlet heat kernel $p_{\sigma = \infty}(t;\cdot,\cdot)$ is equal to zero in the vertices).

Now, starting with~\eqref{BracketingI}, we conclude existence of a constant $C > 0$ such that, also employing Theorem~\ref{WeylsLaw}, 

 %
 %
 %
  \begin{equation*}
    p_{\tau \sigma}\bigg(\frac{1}{\lambda_N(\infty)};x,x\bigg) \leq C \cdot N\ ,
 \end{equation*}
 uniformly in $x \in \mathcal{G}$. From this we immediately obtain the statement.
\end{proof}
\begin{remark} Compared to Lemma~\ref{Lemma1}, we proved something stronger in \cite{BifulcoKerner} and in the finite setting. Namely, we showed that the eigenfunctions of the considered Schrödinger operator on a finite graph are uniformly bounded! On an infinite graph $\mathcal{G}=(\mV_{\mathcal{G}},\mE_{\mathcal{G}})$ with $\inf_{\me \in \mE}  \ell_\me=0$ (where $\ell_\me > 0$ denotes the edge length of $\me \in \mE_{\mathcal{G}}$) this is not necessarily the case. For example, on an interval $\hat{I}$ with length $\hat{L} > 0$ and Dirichlet boundary conditions, any normalized eigenfunction of the Laplacian has supremum $\sqrt{\frac{2}{\hat{L}}}$. So, this supremum diverges for $\hat{L} \rightarrow 0$.
\end{remark}
Going back to~\eqref{EquationOneProof}, Lemma~\ref{Lemma1} in combination with Tonelli's theorem gives
\begin{equation*}\begin{split}
  \lim_{N \rightarrow \infty}  \frac{1}{N}\sum_{n=1}^{N}\left(\lambda_n(\sigma)-\lambda_n(0)\right)&=\sum_{m=1}^{\infty}\sigma_m\int_{0}^{1} \left(\lim_{N \rightarrow \infty} \frac{1}{N}\sum_{n=1}^{N}|f^{\tau \sigma}_n(\mv_m)|^2\right)\ \mathrm{d}\tau \ ,
    \end{split}
\end{equation*}
and hence the proof is complete by taking into account another auxiliary result which is interesting in its own right.
\begin{lemma}[Local Weyl law]\label{LocalWeyl} For arbitrary $\sigma=(\sigma_n)_{n \in \mathbb{N}} \subset [0,\infty)$ one has 
\begin{equation*}
    \lim_{N \rightarrow \infty} \frac{1}{N}\sum_{n=1}^{N}|f^{\sigma}_n(x)|^2=\frac{2}{L \deg(x)}\ ,
\end{equation*}
for all $x \in [0,L)$, where $\deg(x)=2$ whenever $x \neq 0$ and $\deg(x) = 1$ else.
\end{lemma}
\begin{proof}[Proof of Lemma~\ref{LocalWeyl}] The proof employs, once again, the integral kernel $p_{\sigma}(t;\cdot,\cdot)$ of the operator $\mathrm{e}^{-H_{\sigma}t}$ on $\mathcal{G}$. More explicitly, from \eqref{BracketingI} in combination with \cite[Proposition~8.1]{BER} we immediately obtain
	\begin{align}\label{eq:short-time-asymptotics-heat kernel}
	p_{\sigma}(t;x,x) \sim \frac{1}{\sqrt{4\pi t}} \frac{2}{\deg(x)} \:\: \text{as} \:\: t \rightarrow 0^+\ .
	\end{align}
	Note that one also has the expansion 
	\begin{equation*}
	\sum_{n=1}^{\infty}\mathrm{e}^{-\lambda_n(\sigma) t}|f^{\sigma}_n(x)|^2=p_{\sigma}(t;x,x)\ .
	\end{equation*}
	Now, as in~\cite{BHJ,BifulcoKerner}, we can employ Karamata's Tauberian theorem to conclude
	\begin{align}\label{agn:vorstufeocalweyllaw}
	\sum_{\lambda_n(\sigma) \leq \lambda} |f^{\sigma}_n(x)|^2 \sim  \frac{2}{\pi \deg(x)}\sqrt{\lambda}\:\: \text{as} \:\: \lambda \rightarrow \infty\ .
	\end{align}
	Finally, we can employ Theorem~\ref{WeylsLaw} to replace the condition $\lambda_n(\sigma) \leq \lambda$ by $n \leq N$ while $(\frac{\pi N}{\mathcal{L}})^2 \sim \lambda$. This gives the statement.
\end{proof}

\section{A modified local Weyl law}\label{ModifiedWeyl}

In this section we want to illustrate that indeed new interesting features do appear on some infinite quantum graphs. For instance, in \cite{EggerSteinerInfinite} (see also \cite{AKM} for related models) the authors studied the Hamiltonian
\begin{equation}\label{Egger}
H_{\sigma}=-\frac{\mathrm{d}^2}{\mathrm{d}x^2}+\sigma\sum_{n=1}^{\infty}\delta(x-\mv_n)\ ,
\end{equation}
 on $L^2(\mathbb{R}_+)$ for $\sigma \geq 0$ and with $\mv_n=\sum_{m=1}^{n} \frac{\pi}{m}$; this implies that the associated path graph indeed has an infinite total length. They proved that $H_{\sigma}$ still has a purely discrete spectrum (see also the Molchanov-type theorem established in \cite{AKM}) but, at least for $\sigma=\infty$ which refers to Dirichlet boundary conditions in the vertices $\mv_n$, it exhibits non-standard Weyl asymptotics! Of course, in their setting, one is not able anymore to compare the eigenvalues $\left(\lambda_n(\sigma)\right)_{n \in \mathbb{N}}$ with the Neumann ``eigenvalues'' (meaning that $\sigma=0$) since the corresponding Neumann Laplacian is defined over a half-line and therefore does not possess a discrete spectrum. Nevertheless, one could still compare the eigenvalues for two different choices of $\sigma > 0$. 

    In any case, the non-standard Weyl asymptotics lead to a  modified local Weyl law as presented in the following theorem.
    \begin{theorem}[Modified local Weyl law]\label{ModifiedWeylTheorem} Consider the Hamiltonian~\eqref{Egger} with $\sigma=\infty$ over the infinite path graph of infinte length. Then, for all $x \in (0,\infty)$ with $x\neq \mv_n$, one has
    \begin{equation*}
    \lim_{N \rightarrow \infty} \frac{1}{\sqrt{\lambda(N)}}\sum_{n=1}^{N}|f^{\sigma=\infty}_n(x)|^2=\frac{1}{\pi}\ .
\end{equation*}
Here $\lambda=\lambda(N)$ is the function implicitly defined via the relation $N=\frac{\sqrt{\lambda}}{2}\ln \lambda$. On the other hand, for $x=\mv_n$, the above sum is equal to zero.
    \end{theorem}
\begin{proof} The proof is similar to the one of Lemma~\ref{LocalWeyl}. In a first step one realizes that the local heat kernel asymptotics are the unchanged since locally one is on a (maybe small) interval with Dirichlet boundary conditions, i.e., 
	\[
	p_{\sigma=\infty}(t;x,x) \sim \frac{1}{\sqrt{4\pi t}}\:\: \text{as} \:\: t \rightarrow 0^+.
	\]
	Then, once again one employs the expansion 
	\begin{equation*}
	\sum_{n=1}^{\infty}\mathrm{e}^{-\lambda_n(\infty) t}|f^{\sigma=\infty}_n(x)|^2=p_{\sigma=\infty}(t;x,x)\ ,
	\end{equation*}
	as well as Karamata's Tauberian theorem to conclude
	\begin{align}\label{agn:vorstufeocalweyllaw}
	\sum_{\lambda_n(\infty) \leq \lambda} |f^{\sigma=\infty}_n(x)|^2 \sim  \frac{1}{\pi}\sqrt{\lambda}\:\: \text{as} \:\: \lambda \rightarrow \infty\ .
	\end{align}
	In a final step, we use the non-standard Weyl asymptotics derived in \cite[(39)]{EggerSteinerInfinite},
 $$N_{\sigma=\infty}(\lambda) \sim \frac{\sqrt{\lambda}}{2}\ln \lambda\ , $$
 to obtain
	\begin{align*}\label{agn:vorstufeocalweyllaw}
	\frac{1}{\sqrt{\lambda(N)}}\sum_{n \leq N} |f^{\sigma=\infty}_n(x)|^2 &\sim  \frac{1}{\pi}\:\: \text{as} \:\: N \rightarrow \infty\ . \qedhere
	\end{align*}
\end{proof}
Comparing Theorem~\ref{ModifiedWeylTheorem} with Lemma~\ref{LocalWeyl} we conclude that, for the infinite path graph with infinite total length, one has
    \begin{equation*}
    \lim_{N \rightarrow \infty} \frac{1}{N}\sum_{n=1}^{N}|f^{\sigma=\infty}_n(x)|^2=\lim_{\lambda \rightarrow \infty} \frac{1}{N_{\sigma = \infty}(\lambda)} \sum_{\lambda_n(\infty) \leq \lambda} \vert f_n^{\sigma=\infty}(x) \vert^2 = \lim_{\lambda \rightarrow \infty} \frac{2}{\pi \ln \lambda} = 0\ .
\end{equation*}
%
This again reflects the fact that the total length is infinite; more precisely, this should be compared with the right-hand side in Lemma~\ref{LocalWeyl}.
\section{Main Results II: Generalizing the results to more general (infinite) graphs}\label{MainResultsII}

In this section we outline a generalization of Theorem~\ref{MainResult0} and Theorem~\ref{MainResult} to a larger class of infinite quantum graphs, again of finite total length. More explicitly, we consider (infinite) quantum graphs that are constructed as follows: Let $\mathcal{G}_0=(\mV_0,\mE_0)$ be a finite, connected and compact metric graph of total length $\mathcal{L} > 0$ which we also assume to be, without loss of generality, loop-free. The graph Hilbert space is then given by 
\begin{equation}
    L^2(\mathcal{G}_0)=\bigoplus_{\me \in \mE_0}L^2(0,\ell_{\me})
\end{equation}
where $\ell_\me > 0$ denotes the length of the edge $\me \in \mE_0$. Let $q \in L^{\infty}(\mathcal{G}_0) \cap \bigoplus_{\me \in \mE_0}C^{\infty}(0,\ell_{\me})$ be a non-negative, real-valued potential on the graph $\mathcal{G}_0$ and $\gamma =(\gamma_\mv)_{\mv \in \mV_0}\in \mathbb{R}^{|\mV_0|}$ a vector with $\gamma_{\mv} \geq 0$. Then, we introduce the self-adjoint operator $\mathcal{H}^{q,\gamma}_{\mathcal{G}_0}$ via the quadratic form
\begin{equation}\label{FormNeumann}
h^{q,\gamma}_{\mathcal{G}_0}[f]=\int_{\mathcal{G}_0} \vert f^{\prime} \vert^2\ \mathrm{d}x + \int_{\mathcal{G}_0} q \vert f \vert^2\ \mathrm{d}x+\sum_{\mv \in \mV_0}\gamma_\mv|f(\mv)|^2
\end{equation}
for $f \in H^{1}(\mathcal{G}_0)$. Here, $H^1(\mathcal{G}_0) = \bigoplus_{\me \in \mE_0} H^1(0,\ell_\me) \cap C(\mathcal{G}_0)$ with $C(\mathcal{G}_0)$ referring to the space of continuous functions on $\mathcal{G}_0$ (the functions in $C(\mathcal{G}_0)$ are continuous across the vertices). Then, starting with the finite quantum graph $\mathcal{G}_0$, we add a countable number of vertices to the vertex set $\mV_0$ such that the set of additional vertices $\hat{\mV}$ is a collection of isolated points and the new graph $\mathcal{G}$ has an infinite vertex and edge set. More explicitly, we assume that the new vertices $\hat{\mV}$ are contained in the set $\bigcup_{\me \in \mE_0}(0,\ell_\me)$ and we further assume that each cluster point in the set of vertices $\mV_\mathcal{G} := \mV_0 \cup \hat{\mV}$ of $\mathcal{G}$ belongs to $\mV_0$ (in particular, there are only finitely many cluster points) and that it corresponds to a standard vertex, i.e., to some vertex $\mv \in \mV_0$ such that $\gamma_\mv = 0$, cf.\ Figures \ref{fig:y-graph} and \ref{fig:local-graph}.

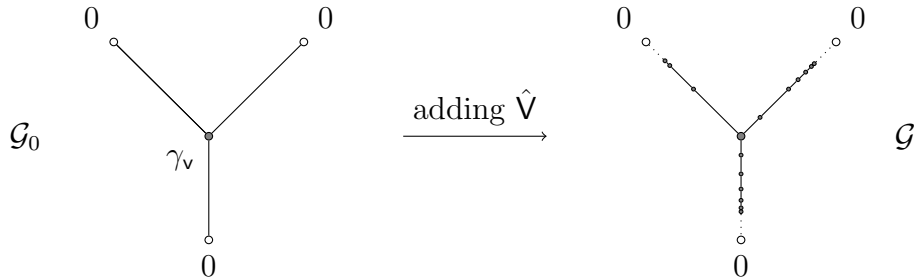
\begin{figure}[h]
\begin{tikzpicture}[scale=0.50]
      \tikzset{enclosed/.style={draw, circle, inner sep=0pt, minimum size=.10cm, fill=gray}}

      \node[enclosed, label={below left: $\gamma_\mv$}] (A) at (0,0) {};
      \node[enclosed, label={above right: $0$}, fill = white] (B) at (2.5,2.5) {};
      \node[enclosed, label={above left: $0$}, fill = white] (C) at (-2.5,2.5) {};
      \node[enclosed, label={below: $0$}, fill = white] (D) at (0,-2.75) {};
      \node[enclosed, white] (X1) at (5,0) {};
      \node[enclosed, white] (X2) at (9,0) {};
      \node[enclosed, label={left: $\mathcal{G}_0$},white] (X3) at (-4,0) {};
      \node[enclosed, label={left: $\mathcal{G}$}, white] (X4) at (19,0) {};

      \node[enclosed] (A') at (14,0) {};
      \node[enclosed, label={above right: $0$}, fill = white] (B') at (16.5,2.5) {};
      \node[enclosed, label={above left: $0$}, fill = white] (C') at (11.5,2.5) {};
      \node[enclosed, label={below: $0$}, fill = white] (D') at (14,-2.75) {};

      \node[enclosed, minimum size=.05cm] (A1) at (14.5,0.5) {};
      \node[enclosed, minimum size=.05cm] (A2) at (15.25,1.25) {};
      \node[enclosed, minimum size=.05cm] (A3) at (15.5,1.5) {};
      \node[enclosed, minimum size=.05cm] (A4) at (15.7,1.7) {};
      \node[enclosed, minimum size=.05cm] (A5) at (15.85,1.85) {};
      \node[enclosed, minimum size=.05cm] (A6) at (15.925,1.925) {};

      \node[enclosed, minimum size=.05cm] (B1) at (12.75,1.25) {};
      \node[enclosed, minimum size=.05cm] (B2) at (12.125,1.875) {};
      \node[enclosed, minimum size=.05cm] (B3) at (12,2) {};

      \node[enclosed, minimum size=.05cm] (C1) at (14,-0.5) {};
      \node[enclosed, minimum size=.05cm] (C2) at (14,-1) {};
      \node[enclosed, minimum size=.05cm] (C3) at (14,-1.4) {};
      \node[enclosed, minimum size=.05cm] (C4) at (14,-1.7) {};
      \node[enclosed, minimum size=.05cm] (C5) at (14,-1.9) {};
      \node[enclosed, minimum size=.05cm] (C6) at (14,-2) {};

      \draw (A) -- (B) node[midway, above] (edge1) {};
      \draw (A) -- (C) node[midway, above] (edge2) {};
      \draw (A) -- (C) node[midway, above] (edge3) {};
      \draw (A) -- (D) node[midway, above] (edge4) {};
      \draw (X1) edge[->] node[above] {\text{adding $\hat{\mV}$}} (X2) node[midway, above] (helpedge) {};

      \draw (A') -- (A1) node[midway, above] (redge1) {};
      \draw (A1) -- (A2) node[midway, above] (redge2) {};
      \draw (A2) -- (A3) node[midway, above] (redge3) {};
      \draw (A3) -- (A4) node[midway, above] (redge4) {};
      \draw (A4) -- (A5) node[midway, above] (redge5) {};
      \draw (A5) -- (A6) node[midway, above] (redge6) {};
      \draw (A6) edge[dotted] (B') node[midway, above] (redge4) {};

      \draw (A') -- (C1) node[midway, above] (dedge1) {};
      \draw (C1) -- (C2) node[midway, above] (dedge2) {};
      \draw (C2) -- (C3) node[midway, above] (dedge3) {};
      \draw (C3) -- (C4) node[midway, above] (dedge4) {};
      \draw (C4) -- (C5) node[midway, above] (dedge5) {};
      \draw (C5) -- (C6) node[midway, above] (dedge5) {};
      \draw (C6) edge[dotted] (D') node[midway, above] (dedge7) {};

      \draw (A') -- (B1) node[midway, above] (ledge1) {};
      \draw (B1) -- (B2) node[midway, above] (ledge2) {};
      \draw (B2) -- (B3) node[midway, above] (ledge3) {};
      \draw (B3) edge[dotted] (C') node[midway, above] (ledge3) {};
     \end{tikzpicture}
     \caption{An infinite $Y$-graph $\mathcal{G}$ with cluster points lying on the boundary.}\label{fig:y-graph}
     \end{figure}
     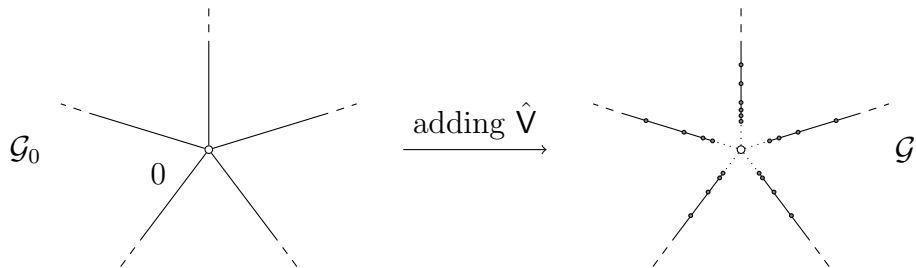
\begin{figure}[h]
\begin{tikzpicture}[scale=0.5]
      \tikzset{enclosed/.style={draw, circle, inner sep=0pt, minimum size=.10cm, fill=gray}, every loop/.style={}}

      \node[enclosed, fill=white, label={below left: $0\:\:\:\:$}] (Z) at (0,0) {};
      \node[enclosed, white] (A) at (0,3) {};
      \node[enclosed, white] (B) at (-3.25,1) {};
      \node[enclosed, white] (C) at (3.25,1) {};
      \node[enclosed, white] (D) at (-1.875,-2.5) {};
      \node[enclosed, white] (E) at (1.875,-2.5) {};
      \node[enclosed, white] (A') at (0,4) {};
      \node[enclosed, white] (B') at (-4,1.25) {};
      \node[enclosed, white] (C') at (4,1.25) {};
      \node[enclosed, white] (D') at (-2.425,-3.25) {};
      \node[enclosed, white] (E') at (2.425,-3.25) {};
      \node[enclosed, white] (A'') at (0.5,2.5) {};
      \node[enclosed, white] (ZA) at (0.5,0.75) {};
      \node[enclosed, white] (B'') at (-2.5,1.25) {};
      \node[enclosed, white] (ZB) at (-1,0.8125) {};
      \node[enclosed, white] (C'') at (2.5,0.25) {};
      \node[enclosed, white] (ZC) at (1,-0.1875) {};
      \node[enclosed, white] (E'') at (1,-2) {};
      \node[enclosed, white] (ZE) at (0,-0.6666) {};
      \node[enclosed, white] (D'') at (-2,-2) {};
      \node[enclosed, white] (ZD) at (-1,-0.6666) {};
      
      \node[enclosed, white] (X1) at (5,0) {};
      \node[enclosed, white] (X2) at (9,0) {};
      
      \node[enclosed, fill=white] (T) at (14,0) {};
      \node[enclosed, white] (F) at (14,3) {};
      \node[enclosed, minimum size=.05cm] (Fdotted) at (14,0.75) {};
      \node[enclosed, white] (G) at (10.75,1) {};
      \node[enclosed, minimum size=.05cm] (Gdotted) at (13.25,0.23077) {};
      \node[enclosed, white] (H) at (17.25,1) {};
      \node[enclosed, minimum size=.05cm] (Hdotted) at (14.75,0.23077) {};
      \node[enclosed, white] (I) at (12.125,-2.5) {};
      \node[enclosed, minimum size=.05cm] (Idotted) at (13.525,-0.625) {};
      \node[enclosed, white] (J) at (15.875,-2.5) {};
      \node[enclosed, minimum size=.05cm] (Jdotted) at (14.475,-0.625) {};
      \node[enclosed, white] (F') at (14,4) {};
      \node[enclosed, white] (G') at (10,1.25) {};
      \node[enclosed, white] (H') at (18,1.25) {};
      \node[enclosed, white] (I') at (11.575,-3.25) {};
      \node[enclosed, white] (J') at (16.425,-3.25) {};
      \node[enclosed, white] (F'') at (14.5,2.5) {};
      \node[enclosed, white] (TF) at (14.5,0.75) {};
      \node[enclosed, white] (G'') at (11.5,1.25) {};
      \node[enclosed, white] (TG) at (13,0.8125) {};
      \node[enclosed, white] (H'') at (16.5,0.25) {};
      \node[enclosed, white] (TH) at (15,-0.1875) {};
      \node[enclosed, white] (I'') at (15,-2) {};
      \node[enclosed, white] (TI) at (14,-0.6666) {};
      \node[enclosed, white] (J'') at (12,-2) {};
      \node[enclosed, white] (TJ) at (13,-0.6666) {};

      \node[enclosed, label={left: $\mathcal{G}_0$},white] (X3) at (-4,0) {};
      \node[enclosed, label={left: $\mathcal{G}$}, white] (X4) at (19,0) {};

	\draw[dashed] (A) edge node[above] {} (A') node[midway, above] (dotted1) {};
      \draw[dashed] (B) edge node[above] {} (B') node[midway, above] (dotted2) {};
      \draw[dashed] (C) edge node[above] {} (C') node[midway, above] (dotted3) {};
      \draw[dashed] (D) edge node[above] {} (D') node[midway, above] (dotted4) {};
      \draw[dashed] (E) edge node[above] {} (E') node[midway, above] (dotted5) {};
      \draw (Z) edge node[right] {} (A) node[midway, above] (edge1) {};
      \draw (Z) edge node[above] {} (B) node[midway, above] (edge2) {};
      \draw (Z) edge node[above] {} (C) node[midway, above] (edge3) {};
      \draw (Z) edge node[above] {} (D) node[midway, above] (edge4) {};
      \draw (Z) edge node[above] {} (E) node[midway, above] (edge5) {};

      \draw (X1) edge[->] node[above] {\text{adding $\hat{\mV}$}} (X2) node[midway, above] (helpedge) {};

      \draw[dashed] (F) edge node[above] {} (F') node[midway, above] (dotted1) {};
      \draw[dashed] (G) edge node[above] {} (G') node[midway, above] (dotted2) {};
      \draw[dashed] (H) edge node[above] {} (H') node[midway, above] (dotted3) {};
      \draw[dashed] (I) edge node[above] {} (I') node[midway, above] (dotted4) {};
      \draw[dashed] (J) edge node[above] {} (J') node[midway, above] (dotted5) {};
      \draw[dotted] (T) edge node[right] {} (Fdotted) node[midway, above] (edge1dotted) {};
      \draw (F) edge node[right] {} (Fdotted) node[midway, above] (edge1) {};
      \draw[dotted] (T) edge node[above] {} (Gdotted) node[midway, above] (edge2dotted) {};
      \draw (G) edge node[right] {} (Gdotted) node[midway, above] (edge2) {};
      \draw (Hdotted) edge node[above] {} (H) node[midway, above] (edge3) {};
      \draw[dotted] (T) edge node[above] {} (Hdotted) node[midway, above] (edge3dotted) {};
      \draw (Idotted) edge node[above] {} (I) node[midway, above] (edge4) {};
      \draw[dotted] (T) edge node[above] {} (Idotted) node[midway, above] (edge4dotted) {};
      \draw (Jdotted) edge node[above] {} (J) node[midway, above] (edge5) {};
      \draw[dotted] (T) edge node[above] {} (Jdotted) node[midway, above] (edge5dotted) {};

      \node[enclosed, minimum size=.05cm] (Fdotted1) at (14,2.25) {};
      \node[enclosed, minimum size=.05cm] (Fdotted2) at (14,1.75) {};
      \node[enclosed, minimum size=.05cm] (Fdotted3) at (14,1.25) {};
      \node[enclosed, minimum size=.05cm] (Fdotted4) at (14,1.05) {};
      \node[enclosed, minimum size=.05cm] (Fdotted5) at (14,0.9) {};
      
      \node[enclosed, minimum size=.05cm] (Gdotted1) at (11.5,0.76923) {};
      \node[enclosed, minimum size=.05cm] (Gdotted2) at (12.5,0.461538461) {};
      \node[enclosed, minimum size=.05cm] (Gdotted3) at (13,0.3076923) {};
      
      \node[enclosed, minimum size=.05cm] (Hdotted1) at (16.5,0.76923) {};
      \node[enclosed, minimum size=.05cm] (Hdotted2) at (15.5,0.461538461) {};
      \node[enclosed, minimum size=.05cm] (Hdotted3) at (15,0.3076923) {};

      \node[enclosed, minimum size=.05cm] (Idotted1) at (12.674998,-1.75) {};
      \node[enclosed, minimum size=.05cm] (Idotted1) at (13.1333305,-1.125) {};
      \node[enclosed, minimum size=.05cm] (Idotted2) at (13.4375,-0.75) {};
      
      \node[enclosed, minimum size=.05cm] (Jdotted1) at (15.325002,-1.75) {};
      \node[enclosed, minimum size=.05cm] (Jdotted1) at (14.8666695,-1.125) {};
      \node[enclosed, minimum size=.05cm] (Jdotted2) at (14.5625,-0.75) {};
      
     \end{tikzpicture}
     \caption{A graph $\mathcal{G}_0$ such that $\mathcal{G}$ has an inner cluster point in $\hat{\mV}$ of degree $5$.}\label{fig:local-graph}
     \end{figure}

Moreover, to each vertex $\hat{\mv} \in \hat{\mV}$ we associate a real number $\sigma_{\hat{\mv}} \geq 0$. Consequently, introducing the form 
\begin{equation}
    h^{q,\gamma,\sigma}_{\mathcal{G}}[f]=h^{q,\gamma}_{\mathcal{G}_0}[f]+\sum_{\mv \in \hat{\mV}}\sigma_\mv|f(\mv)|^2
\end{equation}
on the form domain $\mathcal{D}_{\mathcal{G};q,\gamma,\sigma}:= \mathcal{D}_\mathcal{G} := \{f \in H^1(\mathcal{G}_0): h^{q,\gamma,\sigma}_{\mathcal{G}}[f] < \infty \}$, we obtain a self-adjoint operator $\mathcal{H}^{q,\gamma,\sigma}_{\mathcal{G}}$ associated with $h^{q,\gamma,\sigma}_{\mathcal{G}}[\cdot]$ (similar as in Theorem~\ref{TheoremForm}; note that, since the cluster points belong to $\mV_0$, it is an immediate consequence that $\bigoplus_{\me \in \mE_0} C_0^\infty(0,\ell_\me)$ belongs to the form domain $\mathcal{D}_\mathcal{G}$ and hence it is densely defined).

Moreover, this operator has again purely discrete spectrum and we shall denote its $n$-th eigenvalue by $\lambda_n(q,\gamma,\sigma)$. Furthermore, the corresponding eigenvalue counting function again satisfies the classical Weyl law as formulated in Theorem~\ref{WeylLaw}. Now, to compare the spectrum of the self-adjoint operator $\mathcal{H}^{q,\gamma,\sigma}_{\mathcal{G}}$ with that of $\mathcal{H}^{q,\gamma}_{\mathcal{G}_0}$, we introduce the three quantities
\begin{equation}\label{eq:cesaro-1}
    \frac{1}{N}\sum_{n=1}^{N}\left(\lambda_n(q,\gamma,\sigma)-\lambda_n(q,\gamma,0)\right)\ ,
\end{equation}
\begin{equation}\label{eq:cesaro-2}
    \frac{1}{N}\sum_{n=1}^{N}\left(\lambda_n(q,\gamma,\sigma)-\lambda_n(0,\gamma,0)\right)\ ,
\end{equation}
\begin{equation}\label{eq:cesaro-3}
    \frac{1}{N}\sum_{n=1}^{N}\left(\lambda_n(q,\gamma,\sigma)-\lambda_n(0,0,0)\right)\ .
\end{equation}
We obtain the following statements.

\begin{theorem}\label{GenI} Let $\mathcal{G}$ be an infinite quantum graph as constructed above. Let $\mathcal{H}^{q,\gamma,\sigma}_{\mathcal{G}}$ be the associated Schrödinger operator with $(\sigma_{\hat{\mv}})_{\hat{\mv} \in \hat{\mV}} \notin \ell^1(\hat{\mV})$. Then
\begin{equation*}
\lim_{N \rightarrow \infty} \frac{1}{N}\sum_{n=1}^{N}\left(\lambda_n(q,\gamma,\sigma)-\lambda_n(q,\gamma,0)\right)=\infty\ ,
\end{equation*}
and
\begin{equation*}
\lim_{N \rightarrow \infty} \frac{1}{N}\sum_{n=1}^{N}\left(\lambda_n(q,\gamma,\sigma)-\lambda_n(0,\gamma,0)\right)=\infty\ ,
\end{equation*}
as well as 
\begin{equation*}
\lim_{N \rightarrow \infty} \frac{1}{N}\sum_{n=1}^{N}\left(\lambda_n(q,\gamma,\sigma)-\lambda_n(0,0,0)\right)=\infty\ .
\end{equation*}
\end{theorem}
\begin{proof} The proof of the first identity goes along the same path as the proof of Theorem~\ref{MainResult0}. It combines a monotonicity argument with the main result obtained in \cite{BifulcoKerner} for Schrödinger operators on finite graphs. Note that the regularity assumptions on the potential $q$ are in accordance with \cite{BifulcoKerner}. Moreover, the latter identities follow immediately from the first one, as the expressions in \eqref{eq:cesaro-2} and \eqref{eq:cesaro-3} clearly dominate the one in \eqref{eq:cesaro-1}. 
\end{proof}

\cite[Theorem~1]{BifulcoKerner} also allows us to deduce a counterpart of Theorem \ref{MainResult} for $\ell^1$-summable weights $(\sigma_{\hat{\mv}})_{\hat{\mv} \in \hat{\mV}}$.
\begin{theorem}\label{GenII} Let $\mathcal{G}$ be an infinite quantum graph as constructed above. Let $\mathcal{H}^{q,\gamma,\sigma}_{\mathcal{G}}$ be the associated Schrödinger operator with $(\sigma_{\hat{\mv}})_{\hat{\mv} \in \hat{\mV}} \in \ell^1(\hat{\mV})$. Then
\begin{equation}\label{eq:mwf-1}
\lim_{N \rightarrow \infty} \frac{1}{N}\sum_{n=1}^{N}\left(\lambda_n(q,\gamma,\sigma)-\lambda_n(q,\gamma,0)\right)=\frac{\sum_{\hat{\mv} \in \hat{\mV}}\sigma_{\hat{\mv}}}{\mathcal{L}}\ ,
\end{equation}
and
\begin{equation}\label{eq:mwf-2}
\lim_{N \rightarrow \infty} \frac{1}{N}\sum_{n=1}^{N}\left(\lambda_n(q,\gamma,\sigma)-\lambda_n(0,\gamma,0)\right)=\frac{\int_{\mathcal{G}_0} q \mathrm{d}x+\sum_{\hat{\mv} \in \hat{\mV}}\sigma_{\hat{\mv}}}{\mathcal{L}}\ ,
\end{equation}
as well as
\begin{equation}\label{eq:mwf-3}
\lim_{N \rightarrow \infty} \frac{1}{N}\sum_{n=1}^{N}\left(\lambda_n(q,\gamma,\sigma)-\lambda_n(0,0,0)\right)=\frac{\int_{\mathcal{G}_0} q \mathrm{d}x}{\mathcal{L}}+\frac{2}{\mathcal{L}}\sum_{\mv \in \mV_0}\frac{\gamma_{\mv}}{\deg({\mv})}+\frac{\sum_{\hat{\mv} \in \hat{\mV}}\sigma_{\hat{\mv}}}{\mathcal{L}}\ .
\end{equation}
\end{theorem}
\begin{proof} Up to notation, the proof of \eqref{eq:mwf-1} essentially follows Section~\ref{ProofMainResult}. Taking into account the results of \cite{BifulcoKerner} obtained for finite graphs, one can deduce \eqref{eq:mwf-2} and \eqref{eq:mwf-3} directly from~\eqref{eq:mwf-1}: more precisely, with \cite[Theorem~1]{BifulcoKerner}, one obtains
\[
\lim_{N \rightarrow \infty} \frac{1}{N} \sum_{n=1}^N (\lambda_n(q,\gamma,0) - \lambda_n(0,\gamma,0)) = \frac{\int_{\mathcal{G}_0} q \mathrm{d}x}{\mathcal{L}}\ ,
\]
as well as 
\[
\lim_{N \rightarrow \infty} \frac{1}{N} \sum_{n=1}^N (\lambda_n(q,\gamma,0) - \lambda_n(0,0,0)) = \frac{\int_{\mathcal{G}_0} q \mathrm{d}x}{\mathcal{L}} + \frac{2}{\mathcal{L}} \sum_{\mv \in \mV_0} \frac{\gamma_\mv}{\deg(\mv)}\ ,
\]
taking into account that in these cases we are dealing with finite compact graphs (since $\sigma = 0$). 
\end{proof}
\begin{remark}
   \begin{itemize}
   \item[]
   \item[]
       \item[(i)] Instead of the limit considered in \eqref{eq:mwf-2}, one may also look at the average corresponding to the spectral gaps $\lambda_n(q,\gamma,\sigma) - \lambda_n(q,0,0)$, $n \in \mathbb{N}$. Using the same arguments as before, one obtains
       \[
       \lim_{N \rightarrow \infty} \sum_{n=1}^N (\lambda_n(q,\gamma,\sigma) - \lambda_n(q,0,0)) = \frac{2 \sum_{\mv \in \mV_0} \frac{\gamma_\mv}{\deg(\mv)} + \sum_{\hat{\mv} \in \hat{\mV}} \sigma_{\hat{\mv}}}{\mathcal{L}}\ .
       \]
       \item[(ii)] In our considerations, cluster points are elements in the vertex set $\mV_0$ of the underlying compact finite metric graph $\mathcal{G}_0$. However, it is also possible generalize this and to allow for (finitely many!) cluster points lying not in $\hat{\mV}$ but somewhere along the edges of $\mathcal{G}_0$.

       \item[(iii)] In a future work, it might be interesting to translate our results to other infinite quantum graphs, for example, as considered in \cite{MugnoloTaeuferInfinite}. In particular, they study so-called comb graphs (see~\cite[Figure~3.1]{MugnoloTaeuferInfinite}) and it turns out that such types of graphs are not included in our setting, even with finite length. Hence, to say something more for such graphs, one would need to derive a (possibly modified) Weyl law and to understand heat kernel asymptotics on the diagonal.

   \end{itemize}
\end{remark} 

\subsection*{Acknowledgements}{PB was supported by the Deutsche Forschungsgemeinschaft DFG (Grant 397230547). We are happy to thank our recent visitors to the FernUniversität in Hagen as well as our colleagues for interesting and helpful discussions. We are also happy to thank A.~Kostenko (Ljubljana, Vienna), D.~Mugnolo (Hagen) as well as N. Nicolussi (Vienna) for remarks on an earlier version of the manuscript.}
	
	\vspace*{0.5cm}
	
	{\small
		\bibliographystyle{amsalpha}
		\bibliography{Literature}}

\end{document}